\documentclass[12pt]{amsart}
\usepackage{amssymb,amsmath, mathtools}
\usepackage{colordvi}
\usepackage{graphicx}
\usepackage{verbatim}
\usepackage{preamble, ifthen}   

\title[Differences between ranks\dots]{On differences between the border rank and the smoothable rank of a polynomial}

\begin{document}
 \author[W.~Buczy\'nska]{Weronika Buczy\'nska}
 \address{Weronika Buczy\'nska\\
 Institute of Mathematics of the
 Polish Academy of Sciences\\
 ul. \'Sniadeckich 8\\
 P.O. Box 21\\
 00-956 Warszawa, Poland}
  \email{wkrych@mimuw.edu.pl}

\thanks{W.~Buczy\'nska is supported 
   by the research project
  ``Rangi i rangi brzegowe wielomian\'ow oraz r\'ownania rozmaito\'sci siecznych''
   funded by the Polish government science programme in 2012-2014.	 }

 \author[J.~Buczy\'nski]{Jaros\l{}aw Buczy\'nski}
 \address{Jaros\l{}aw Buczy\'nski \\
 Institute of Mathematics of the
 Polish Academy of Sciences\\
 ul. \'Sniadeckich 8\\
 P.O. Box 21\\
 00-956 Warszawa, Poland}
 \email{jabu@mimuw.edu.pl}

 \thanks{J.~Buczy\'nski is supported by the project
        ``Secant varieties, computational complexity, and toric degenerations''
        carried out under Homing Plus programme of the Foundation for Polish Science,
        funded in part by the European Union, Regional Development Fund. 
        J.~Buczy\'nski is also supported by the scholarship ``START'' of the Foundation for Polish Science.}

\begin{abstract}
   We consider higher secant varieties to Veronese varieties. 
   Most points on the $r$-th secant variety are represented by a finite scheme of length $r$ contained in the Veronese variety
    --- in fact, for generic point, it is just a union of $r$ distinct points.
   A modern way to phrase it is: the smoothable rank is equal to the border rank for most polynomials.
   This property is very useful for studying secant varieties, especially, whenever the smoothable rank is equal to the border rank for all points of the secant variety in question.
   In this note we investigate those special points for which the smoothable rank is not equal to the border rank.
   In particular, we show an explicit example of a cubic in five variables with border rank $5$ and smoothable rank $6$.
   We also prove that all cubics in at most four variables have the smoothable rank equal to the border rank.
\end{abstract}

\maketitle

\date{20 January 2014}

\section{Introduction}

Throughout the paper we work over an algebraically closed field $k$ of characteristic $0$.

Let $X \subset \PP^N$ be an embedded projective variety.
We will later concentrate on the case in which $X$ is a Veronese variety, but for a while we consider a more general situation.
The $r$-th \emph{secant variety} of $X$ is:
\[
  \sigma_r(X) = \bigcup \overline{\set{\langle R \rangle : R = \setfromto{x_1}{x_r}, x_i \in X}} \subset \PP^N,
\]
where $\langle R \rangle$ is the linear span of the finite set $R$. In particular, $\sigma_1(X) = X$. 
Less formally, but in plain English, the $r$-th secant variety of $X$ consists of linear combinations of at most $r$ points on $X$ and their limits.
We emphasise \emph{``\dots and their limits''}. 
It is very difficult to conceive these limits, and in fact, it is unlikely that such understanding can be achieved in general.
For $r=2$, the description of the limits is well known, see Section~\ref{sect_ranks_two} for an overview. The case when $r=3$ and $X$ is a special kind of homogeneous space 
   (\emph{generalised cominuscule}) is treated in \cite{landsberg_jabu_third_secant}.
With the exception of the case in which $X$ is a Veronese variety (or its subvariety), very few other results are known.

Naively, taking a family of points on $\sigma_r(X)$ parametrised by one parameter $t$, say $f(t) \in \langle R(t) \rangle$ with each $R(t) = \setfromto{x_{1}(t)}{x_{r}(t)}$  
  for all $t \ne 0$, 
  one could hope that
\begin{equation}  \label{equ_naive_hope}
  \lim_{t \to 0} f(t) \in \langle \lim_{t \to 0} R(t)\rangle.
\end{equation}
Here $\lim_{t \to 0} R(t)$ means the flat limit of schemes, or, in other words, the limit in the Hilbert scheme,
  and for a scheme $Q \subset X$,  $\langle Q\rangle$ denotes the linear span of the scheme.
Since the Hilbert scheme is projective, this limit always exists.
Unfortunately, $\dim \langle R(t) \rangle$ (for general $t$) is not always equal to $\dim \lim_{t \to 0} \langle R(t) \rangle$,
  and \eqref{equ_naive_hope} may fail to hold.
This motivates the following definitions:

For a point $p \in \PP^N$ denote by $\borderrk_X(p)$ the $X$-\emph{border rank} of $p$, that is the minimal $r$ such that $p \in \sigma_r(X)$. 
This definition is fairly standard, see for instance \cite{landsberg_teitler_ranks_and_border_ranks_of_symm_tensors}, \cite{landsberg_border_rank_of_multiplication}.

The $X$-\emph{smoothable rank} of $p \in \PP^N$ is denoted by $sr_X(p)$ and it is the minimal integer $r$ 
  such that there exists a finite scheme $R \subset X$ of length $r$ which is smoothable in $X$,
  with $p \in \langle R \rangle$.
This definition appeared in \cite{ranestad_schreyer_on_the_rank_of_a_symmetric_form}, 
  and was motivated by the results in \cite{bernardi_gimigliano_ida}, \cite{jabu_ginensky_landsberg_Eisenbuds_conjecture}, \cite{nisiabu_jabu_cactus}.
The $X$-\emph{cactus rank} is another variant considered in \cite{ranestad_schreyer_on_the_rank_of_a_symmetric_form}, 
  \cite{bernardi_ranestad_cactus_rank_of_cubics}. We define it in Section~\ref{sect_general_variety},
  however in most cases considered in this article, the $X$-cactus rank is equal to the $X$-smoothable rank.

We always have $\borderrk_X(p) \le sr_X(p)$.
Somehow the points $p$ for which $\borderrk_X(p)= sr_X(p)$ are ``easier'' to treat.
Particularly, whenever $\borderrk_X= sr_X$ for all points in $\PP^N$, or at least for all points in $\sigma_r(X)$ for some $r$,
  the secant varieties are more ``tame'', see for instance \cite{bernardi_gimigliano_ida}, \cite{nisiabu_jabu_cactus}, \cite{jabu_ginensky_landsberg_Eisenbuds_conjecture}.
The purpose of this note is to present a few examples when $\borderrk_X(p) < sr_X(p)$, which we call ``wild'' examples. 

Note that the definition of $X$-smoothable rank \emph{does not} involve limits, at least not in any direct way.
One may argue that the definition of smoothability does involves limit, but in many cases of interest,
   all finite subschemes of $X$ of a given length are in a single irreducible component of the Hilbert scheme, 
   and they are all smoothable.
Informally, we say that the smoothable rank is an (imperfect) way 
   to get rid of limits when considering the secant varieties.

When $\borderrk_X(p) \le 2$, then $sr_X(p) = \borderrk_X(p)$  (see Section~\ref{sect_ranks_two}). 
Already when $\borderrk_X(p) = 3$, it is possible to construct explicit examples of $X$ and $p$ with $sr_X(p) > 3$ (see Section~\ref{sect_ranks_three}).
The situation becomes  more interesting, if we limit ourselves to the case in which $X$ is a Veronese variety.

Let $V$ be a vector space of dimension $n$. We consider homogeneous polynomials of degree $d$
  in $n$ variables, which form a basis of $V$. That is, we consider elements of $S^d V$.
Let $v_d(\PP V) \subset \PP (S^d V)$ be the Veronese variety:
\[
   v_d(\PP V) = \set{[v^d] \in \PP (S^d V) : v \in V}.
\]
This proposition follows easily from general knowledge and published articles, see Section~\ref{sect_tameness_principle} for a discussion:
\begin{prop}\label{prop_low_border_rank_is_tame}
   For  $X = v_d(\PP V) \subset \PP (S^d V)$ and $f \in S^d V$ of $X$-border rank at most $\max (4, d+1)$,
     we always have $sr_X([f]) = \borderrk_X([f])$ (all such polynomials are ``tame'').
\end{prop}

We prove the following theorem for cubic polynomials.
\begin{thm}\label{thm_wild_and_tame_cubics}
   Consider $X = v_3(\PP V) \subset \PP (S^3 V)$.
   \begin{itemize}
     \item If $\dim V \le 4$, then $sr_X([f]) = \borderrk_X([f])$ for all $f \in S^3 V$ (all such polynomials are ``tame'').
     \item If $\dim V \ge 5$, then a polynomial $f\in S^3 V$ with 
           $sr_X([f]) > \borderrk_X([f])$ exists, or less formally, ``wild'' polynomials exist.
   Specifically, if $\dim V =5$ and $\set{x_0, x_1, y_0, y_1, y_2}$ is a basis of $V$,
      then 
   \[
      f= x_0^2 \cdot y_0 - (x_0 + x_1)^2 \cdot y_1 + x_1^2 \cdot y_2 
   \]
   is ``wild'', with $\borderrk_X([f]) = 5$ and $sr_X([f]) =6$. 
   Furthermore, $r(f)$ is $9$.
   \end{itemize}
\end{thm}
The difficulty of Theorem~\ref{thm_wild_and_tame_cubics} lies in the cases $\dim V =4$ and $\dim V =5$.
For $\dim V \le 3$ it follows easily from Proposition~\ref{prop_low_border_rank_is_tame},
  while for $\dim V \ge 6$ the argument is identical to the argument for $\dim V=5$,
  see Section~\ref{sect_higher_dimension_examples}.

Note that the rank of a general cubic in 5 variables is 8.
Thus Theorem~\ref{thm_wild_and_tame_cubics} provides an example of a form, whose rank is higher then the generic rank.
Few such examples are known, except in two or three variables, see \cite{teitler_maximum_rank_of_monomials}.
  
Prior to posting this article, 
  we communicated the content and the proof of Theorem~\ref{thm_wild_and_tame_cubics} for $\dim V =5$ to 
  A.~Bernardi and B.~Mourrain, 
  who applied this result to their work \cite[Ex.~2.8]{bernardi_brachat_mourrain_comparison}.

\subsection*{Acknowledgements}
We thank Alessandra Bernardi, Enrico Carlini, Anthony Iarrobino, Bernard Mourrain and Kristian Ranestad for useful conversations,
and for their questions that motivated us to write this article.
The article is written as a part of "Computational complexity, generalised Waring type problems and tensor decompositions",
a project within "Canaletto", the executive program for scientific and technological cooperation between Italy and Poland, 2013-2015.

\section{General variety}\label{sect_general_variety}
We commence with a brief summary of the case in which $X \subset \PP^N$ is an arbitrary variety (or even a reduced scheme).
Admittedly, the content of this section is not very original but it serves to explain, using small and easy examples, how the anomalies arise.
In the case of polynomials, the same methods are used to produce ``wild'' examples,
  but since the Veronese variety is, in some sense, more ``regular'',
  those wild examples arise much later, and they are more complicated.

The $X$-\emph{cactus rank} of $p \in \PP^N$ is denoted by $cr_X(p)$ and it is the minimal number $r$ 
  such that there exists a finite scheme $R \subset X$ of length $r$ with $p \in \langle R \rangle$.
Here, we consider \textbf{all} finite schemes $R$,
  whereas in the definition of $sr_X$ we only consider those that are \textbf{smoothable} in $X$.

\subsection{Comparing ranks}

For all points,  the following inequalities of functions hold:
\begin{equation}\label{equ_compare_ranks_for_general_X}
\begin{aligned}
   \borderrk_X &\le sr_X \le r_X\\
   cr_X &\le sr_X \le r_X\\
\end{aligned}
\end{equation}
In some special situations we can say more, but in general, all inequalities can be strict, 
   and there is no simple inequality between $cr_X$ and $\borderrk_X$.
Bernardi and Ranestad \cite{bernardi_ranestad_cactus_rank_of_cubics} 
   show that the cactus rank of a general form with respect to the Veronese variety 
   can be strictly smaller than its border rank. 
In particular, this holds for cubics in at least $9$ variables.
On the other hand, Theorem~\ref{thm_wild_and_tame_cubics} produces examples of cubics with 
   $cr_X(F) =sr_X(F) > \borderrk_X(F)$.
Thus cubics in many variables can have either $cr_X$ or $\borderrk_X$ smaller.

Slightly more can be said, when $X$ is the Veronese variety, see Section~\ref{sect_comparing_ranks_for_Veronese}.
We note that one could also define the border versions of smoothable and cactus ranks.
We will not consider these concepts here, we only emphasise that the \emph{border smoothable rank} 
  is simply equal to the usual border rank.
Please see \cite{bernardi_brachat_mourrain_comparison} for more details about comparing various concepts of ranks,
  although the authors limit the discussion to the case when $X$ is the Veronese variety.

\subsection{Overview of tame behaviour for border rank and cactus rank at most two}\label{sect_ranks_two}
  
Every point $p \in \sigma_2(X)$ is one of the following types:
\begin{enumerate}
  \item $p \in X$, or
  \item $p \in \langle x_1, x_2 \rangle$ for two distinct points $x_i \in X$, or
  \item $p$ is in the tangent star to $X$ at a point $x \in X$. 
        See for instance \cite[Section~1.4]{jabu_ginensky_landsberg_Eisenbuds_conjecture} for a definition of the tangent star. 
\end{enumerate}
If $X$ is smooth, then the tangent star is equal to the tangent space (in particular, its dimension is $\dim X$).
Otherwise, dimension of the tangent star at a fixed point is at most $2\dim X$.
It  immediately follows that if $X$-border rank of $p$ is at most $2$, or $sr_X(p)\le 2$, then 
\begin{equation}\label{equ_all_ranks_equal}
  sr_X(p) = cr_X(p) = \borderrk_X(p) = 1 \text{ or } 2.
\end{equation}
See also \cite[Prop.~3.1]{jabu_ginensky_landsberg_Eisenbuds_conjecture}.
If, in addition, $X$ is smooth, or more generally,
  if the tangent star is equal to the projective Zariski tangent space at every point of $X$,
  then $cr_X(p) \le 2$ also implies \eqref{equ_all_ranks_equal}.
However, the first anomaly occurs here whenever the tangent star is not equal to the projective Zariski tangent space.
Specifically, it may happen, that $cr_X(p) =2$, whereas $\borderrk_X(p)$ and $sr_X(p)$ are arbitrarily large (for some choices of $X$ and $p$).
\begin{example}
   Consider $X$ to be a union of $N$ lines in $\PP^N$, all passing through a fixed point $o$.
   If the union of lines is not contained in any hyperplane, then any $p \in \PP^N$ has $X$-cactus rank $1$ or $2$,
       whereas if $N$ is large, the $X$-border rank of a general point in $\PP^N$ is also large. 
   The same holds, if $X\subset \PP^N$ is any curve with a singularity $o\in X$, such that the Zariski tangent spaces are equal:
      $T_o X = T_o \PP^N$. 
\end{example}
\begin{prf}
   We always have $cr_X(p) = 1$ if and only if $p \in X$.
   So suppose $p \notin X$. 
   We claim $cr_X(p) \le 2$, and to prove it we need to construct a scheme $R\subset X$ of length $2$
     such that $p \in \langle R \rangle$.
   Take a line $\ell \subset \PP^N$ containing $p$ and the singularity $o$.
   Inside this line take $R$ to be the unique scheme of length $2$ supported in $o$.
   Clearly, $\langle R \rangle = \ell \ni p$. 
   Moreover, $R \subset X$, since the Zariski tangent space to $X$ at $o$ is $T_o \PP^N$.
   Thus $R$ computes the cactus rank of $p$.

   On the other hand, $\dim \sigma_r(X) \le 2r-1$ (in fact, we have an equality,
     if $X$ is an irreducible curve, but a strict inequality, if $X$ is the union of lines as above and $r >1$).
   Thus the $X$-border rank of a general point $p\in \PP^N$ is at least $\lceil\frac{N+1}{2}\rceil$.
\end{prf}
Precisely this anomaly has been used in \cite[Sections~3.3--3.5]{jabu_ginensky_landsberg_Eisenbuds_conjecture} 
  to produce counter-examples to a conjecture of Eisenbud-Koh-Stillman.

\subsection{Border rank three}\label{sect_ranks_three}
For a smooth $x\in X$ denote by $\PP \hat T_x X \subset \PP^N$ the embedded projective tangent space.
The following proposition is also common knowledge but it may serve to provide ``wild'' examples with border rank $3$. 
\begin{prop}\label{prop_three_colinear_points}
   Suppose there exist three collinear points $x, y,z \in X$ and each of them is a smooth point of $X$.
   Then the linear span of the three tangent spaces is contained in $\sigma_3(X)$:
   \[
     \langle \PP \hat T_x X, \PP \hat T_y X, \PP \hat T_z X\rangle \subset \sigma_3(X).
   \]
\end{prop}
\begin{prf}
  Pick $\hat x_0, \hat y_0, \hat z_0 \in k^{N+1}$, such that $[\hat x_0] = x$, $[\hat y_0] = y$, and $[\hat z_0] = z$.
  Since $x,y,z$ are collinear,  we may suppose $\hat{x_0} + \hat{y_0} + \hat{z_0} = 0$.
 
  Pick any point $[v]\in \langle \PP \hat T_x X, \PP \hat T_y X, \PP \hat T_z X \rangle$,
    and decompose $v = v_x + v_y + v_z \in k^{N+1}$ so that $[v_x] \in  \PP \hat T_x X \subset \PP^N$, and analogously for $v_y$ and $v_z$.
  Then there exist curves $\hat{x}(t), \hat{y}(t), \hat{z}(t)$ in the affine cone $\hat X$ over $X$,
    such that $\hat{x}(0) = \hat{x}_0$, $\frac{\ud \hat{x}}{\ud t} (0)  = v_x$, and analogously for $\hat{y}$ and $\hat{z}$.
  Take $p(t) = \frac{1}{t} (\hat{x}(t) + \hat{y}(t) + \hat{z}(t))$, and clearly $[p(t)] \in \sigma_3(X)$.
  In particular,
  \begin{align*}
    \sigma_3(X) \ni [p(0)] &= \left[\lim_{t \to 0} \frac{1}{t} (\hat{x}(t) + \hat{y}(t) + \hat{z}(t))\right]\\
                           &= \left[\frac{\ud (\hat{x} + \hat{y} + \hat{z})}{\ud t} (0)\right] = [v].
  \end{align*}
\end{prf}
In the situation of the proposition, if $p \in \langle \PP \hat T_x X, \PP \hat T_y X, \PP \hat T_z X\rangle$,
  then $sr_X(p) \le 6$, and this leaves a possibility for $sr_X(p)$ to be more than $\borderrk_X(p) =3$.
In fact, this happens for some $X$. For instance, if $X = \PP A \times \PP B \times \PP C \subset \PP(A\otimes B \otimes C)$ 
  in the Segre embedding, then the points of border rank $3$ come in four types, see \cite[Thm~1.2]{landsberg_jabu_third_secant}.
The last type, that is 
\[
 p=a_2 \otimes b_1 \otimes c_2 + a_2 \otimes b_2 \otimes c_1 + a_1 \otimes b_1 \otimes c_3 + a_1\otimes b_3 \otimes c_1  + 
   a_3\otimes b_1 \otimes c_1,
\]
has $X$-cactus rank and $X$-smoothable rank equal to $4$, whenever the vectors appearing in the expression are linearly independent. To see that, suppose for contradiction the $X$-cactus rank is less than $4$. 
Then there exists a scheme $R \subset X$ of length at most $3$, with $p \in \langle R \rangle$.
By the considerations in Section~\ref{sect_ranks_two}, length of $R$ is $3$.
Thus $R$ is either a union of three distinct reduced points, or a reduced point and a double point $\Spec k[t]/(t^2)$ contained in some line,
  or a triple point $\Spec k[t]/(t^3)$ contained in some curve. ($R$ cannot be a double point $\Spec k[t,u]/(t^2, tu, u^2)$, since then $cr_X(p)\le 2$, compare with \cite[Lem.~2.3]{nisiabu_jabu_cactus}.)
Each of these three cases, corresponds to the cases (i)--(iii) of \cite[Thm 1.2]{landsberg_jabu_third_secant}.
But $p$ is not in any of these cases by the last sentence of \cite[Thm 1.2]{landsberg_jabu_third_secant}.
So $sr_X(p) \ge cr_X(p)\ge 4$, and they are at most $4$ by \cite[Thm 1.2(iv)]{landsberg_jabu_third_secant}.

More generally, one may construct curves in large projective spaces, such that 
   $sr_X(p) \ge cr_X(p)\ge 4$, while $\borderrk_X(p) =3$ for some point $p$. 
\begin{question}
   Is there a universal bound on $sr_X(p)$ for points  $p \in \sigma_3(X)$?
   That is, does there exist an integer $r$, such that for all $N$, all $X\subset \PP^N$, and all $p \in \sigma_3(X)$,
      we have $sr_X(p) \le r$? 
\end{question}
Of course, as stated above, the points obtained using Proposition~\ref{prop_three_colinear_points} have $sr_X(p) \le 6$.
But there might be other ways to construct wild points. 
In fact, we expect a negative answer to this question,
  but one may limit oneself for example to only smooth $X$, or even only to $X$ which is a homogeneous space, to obtain a sensible bound.

\subsection{Higher border rank}

More generally, a statement analogous to Proposition~\ref{prop_three_colinear_points} for arbitrary number of points, and for singular points is true, and the proof is conceptually identical, only the notation becomes more complicated:
\begin{prop}\label{prop_r_linearly_dependent points}
   Suppose there exist points $\fromto{x_1}{x_r} \in X$, that are linearly degenerate, that is $\dim \langle \fromto{x_1}{x_r} \rangle < r-1$. 
   Then the join of the $r$ tangent cones at these points is contained in $\sigma_r(X)$.
   In the case $X$ is smooth at $\fromto{x_1}{x_r}$:
   \[
     \langle \fromto{\PP \hat T_{x_1} X}{\PP \hat T_{x_r} X} \rangle \subset \sigma_r(X).
   \]
\end{prop}
\noprf

\section{Tame cases for polynomials}
We begin this section by providing some standard references and facts about polynomial decompositions. 
That is we consider the Vero\-ne\-se variety $X=v_d(\PP V) \subset \PP (S^d V$),
   as defined in the introduction, and from now on we simply say \emph{rank}, \emph{border rank}, etc,
   to mean $X$-rank, $X$-border rank etc.
By a standard abuse of notation, we will apply all sorts of rank both to points in $S^d V$ (polynomials) and points in
   $\PP (S^d V)$, as convenient.

\subsection{Conciseness and ranks}   
\begin{defin}
  For a linear subspace $W\subset V$, we say that a polynomial $f \in S^d W$ is \emph{$n$-concise},
  or in other words, $f$ \emph{depends essentially on $n$ variables},
  if $\dim W = n$ is minimal, that is $f \notin S^d U$ for any linear subspace $U \subsetneqq W$.
  We denote this integer $n$ by $concise(f)$. 
\end{defin}

In the situation of the definition, if  $f \in S^d W$ with $W \subset V$ all ranks may be calculated in $W$,
that is:
\begin{equation}\label{equ_calculate_ranks_for_concise}
  cr_{v_d(\PP V)}(f) = cr_{v_d(\PP W)}(f) \text{ and analogously for } \borderrk_X, sr_X, r_X.
\end{equation}

For $r_X$ this is proven in \cite[Rem.~2.3]{carlini_catalisano_geramita_Waring_rank_of_monomials};
  for $\borderrk_X$ an analogous statement is shown for general tensors in \cite[Prop.~2.1 \& Cor.~2.2]{landsberg_jabu_ranks_of_tensors} 
  (the proof for the symmetric case is also analogous).
  For the cases of $cr_X$ and $sr_X$,
  we slightly modify the proof for $r_X$. 
Specifically, if $R \subset \PP V$ is such that $[f] \in \langle v_d(R)  \rangle$,
  then we can take a general linear projection $V \to W$, and let $Q$ be the image of $R$ under this projection.
We claim that $[f] \in \langle v_d(Q) \rangle$. 
Indeed, we have the induced linear projection 
  $\PP (S^d V) \dashrightarrow \PP (S^d W)$, which restricted to the Veronese varieties,
  is just the projection $\PP V \dashrightarrow \PP W$ 
  (in particular, the image of $\langle v_d(R)  \rangle$ is $\langle v_d(Q)  \rangle$),
  and it is the identity on $\PP (S^d W)$ (in particular, the image of $[f]$ is $[f]$).
This proves the statement about $cr_X$, since $\length Q \le \length R$. 
 
If in addition $R$ was smoothable (to calculate $sr_X$), then we can also project the smoothing.
Now the limit of the projected smoothing is $Q'$, and its length is equal to the length of $R$.
Moreover $Q \subset Q'$, so $[f] \in \langle v_d(Q)  \rangle \subset \langle v_d(Q')  \rangle$.
The details of this argument resemble the argument in the proof of 
   \cite[Lem.~2.8]{jabu_ginensky_landsberg_Eisenbuds_conjecture}.

To see that
\begin{equation}\label{equ_concise_is_lower_bound}
   cr_X(f), sr(f), \borderrk_X(f), r_X(f) \ge concise(f)
\end{equation}
we pick a finite scheme $R\subset \PP^{n-1}$ (perhaps smoothable or smooth) such that:
\[
   [f] \in \langle v_d(R)  \rangle \subset \langle v_d (\PP^{n-1}) \rangle =\PP S^d k^n. 
\]
If $f$ is $n$-concise, then $R$ must span $\PP^{n-1}$, which is only possible when $R$ has length at least $n$.
Also being $m$-concise for some $m\le n$ is a closed condition, so the same inequality applies also to $\borderrk_X(f)$.

\subsection{Quadrics}

If $d=2$, then by standard linear algebra for all $f \in S^2 V$ we have $r(f) = concise(f)$. 
So using \eqref{equ_concise_is_lower_bound}, and $cr, \borderrk, sr \le r$,
we conclude $cr(f) = \borderrk(f) = sr(f) = r(f) = concise(f)$, which is just the standard notion of the rank of a quadratic form.

\subsection{Annihilator}\label{sect_annihilator}\label{sect_comparing_ranks_for_Veronese}

We consider the polynomial ring $\Sym V^*$, 
  the coordinate ring of the projective space $\PP(V)$,
  which we identify with the algebra of polynomial differential operators with constant coefficients acting on $\Sym V$.
For $\alpha \in S^i V^*$, $f \in S^d V$, we denote $\alpha \hook f$ the result of the differentiation. 
We let $\Ann(f) \subset \Sym V^*$ be the \emph{annihilator} of $f \in S^d V$ (also called the \emph{apolar ideal} of $f$):
\[
    \Ann(f)  := \set{ \alpha \in S^{\bullet} V^* \mid  \alpha \hook p = 0 }.
\]
Such an ideal (which is a Gorenstein ideal) has been considered by many authors since the time of Macaulay, 
   but recently it has got a lot of attention in relation to the secant varieties and symmetric tensor rank.
See \cite{iarrobino_kanev_book_Gorenstein_algebras}, and \cite[Section~21.2]{eisenbud} for exhaustive reports on this subject,
   and \cite[Section~4]{nisiabu_jabu_cactus} for a brief and comprehensive overview of the fundamental properties.
The main observations that make the annihilator useful to our study are as follows:
\begin{enumerate}
  \item For $f \in S^d V$ and  $R \subset \PP V$, we have the following equivalence:
         $f \in \langle v_d (R) \rangle$, if and only if  the homogeneous ideal of $R$ is contained in $\Ann(f)$.
  \item Let $H_f$ be the Hilbert function of $\Sym V^* / \Ann(f)$.
        Then for any integer $i$ all ranks are at least $H_f(i)$:
        \[
           r(f), \borderrk(f), sr(f), cr(f)\ge H_f(i).
        \]
\end{enumerate}

We add (ii) above to \eqref{equ_compare_ranks_for_general_X}, 
  \eqref{equ_concise_is_lower_bound} to show how various notions of ranks (for homogeneous forms) depend on each other:
\begin{align*}
   H_f(i) \le \borderrk(f) &\le sr(f) \le r(f) \text{ for all $i\in\setfromto{1}{d-1}$},\\
   H_f(i) \le cr(f) &\le sr(f) \le r(f) \text{ for all $i\in\setfromto{1}{d-1}$},\\
   \text{and } H_f(1) &= concise(f).
\end{align*}

\subsection{Two variables and the tameness principle}\label{sect_tameness_principle}

In arbitrary degree, but in two variables ($f$ is $2$-concise), 
  a modern way to phrase the 19th century result of Sylvester is 
  \[
    \borderrk(f) = sr(f) = cr(f).
  \]
The reference is \cite[Thm.~1.44]{iarrobino_kanev_book_Gorenstein_algebras}, 
  and we explain how the result follows from that theorem.
Specifically, $\Ann(f)$ is a complete intersection ideal \cite[Thm.~1.44(iv)]{iarrobino_kanev_book_Gorenstein_algebras}, 
  minimally generated by two functions, say of degrees $d_1, d_2$, with $d_1 \le  d_2$.
Moreover $H_f(i) \le d_1$, and $H_f(d_1-1) = d_1$ \cite[Thm.~1.44(i)]{iarrobino_kanev_book_Gorenstein_algebras}.
Thus all ranks of $f$ are at least $d_1$.
The generator of degree $d_1$ defines a scheme $R$ of length $d_1$,
  whose ideal is contained in the annihilator, which implies that 
  the cactus rank of $f$ is also at most $d_1$, so $cr(f) =d_1$.
Since $R \subset \PP^1$, it is smoothable, thus $sr(f) =d_1$.
Also $\borderrk(f) \le d_1$ by  \cite[Thm.~1.44(ii)]{iarrobino_kanev_book_Gorenstein_algebras}.

The considerations above can be partially generalised to any number of variables in the following form:
\begin{principle}\label{prin_tameness}
   Let $f\in S^d V$ and $X = v_d(\PP V)$. 
   If $\borderrk(f) \le  d+1$, then $sr(f) = \borderrk(f)$.
\end{principle}
For a proof see \cite[Prop.~2.5]{nisiabu_jabu_cactus} or \cite[Prop.~11]{bernardi_gimigliano_ida}.

In particular, Proposition~\ref{prop_low_border_rank_is_tame} follows: 
Applying the considerations above to quadrics we may assume $d\ge 3$, and then Proposition~\ref{prop_low_border_rank_is_tame} follows from Principle~\ref{prin_tameness}.

We will use the following lemma:
\begin{lemma}\label{lem_contained_in_a_line}
   Suppose $\ccI \subset \Sym V^*$ is a homogeneous ideal (not necessarily saturated)
      with $\dim S^d V^* / \ccI_d  \le d+1$,
      and $R \subset \PP V$ is the subscheme defined by $\ccI$.
   Then either $R$ is a finite (or empty) scheme  of length at most $d+1$,
      or there exists a linearly embedded $\PP^1 \subset \PP V$
      such that  $R \subset \PP^1$.
\end{lemma}
\begin{prf}
   First observe that it is sufficient to assume that $\ccI$ is generated by degree $d$: 
     smaller degrees have no effect on $R$,
     whereas the subideal generated by $\ccI_d$ defines $R' \subset \PP V$ with $R \subset R'$,
     and it suffices to prove the claim for $R'$.
   
   Let $A = \Sym V^*/\ccI$.
   The Macaulay bound \cite[Thm~3.3]{green_generic_initial_ideals}, \cite[Thm~2.2(i), (iii)]{stanley_hilbert_functions_of_graded_algebras}
       on the growth of Hilbert function gives $\dim A_{d+1} \le d+2$.
   We consider two cases.
   If $\dim A_{d+1} \le d+1$, then further applications of the Macaulay bound give 
     $\dim A_{i} \le d+1$ for all $i\ge d$.
   So the Hilbert polynomial of $A$ is a constant $r \le d+1$, and $R$ is finite of length at most $d+1$.
   Otherwise $\dim A_{d+1} = d+2$, and the Gotzmann persistence theorem 
      \cite{gotzmann_persistence_theorem}, \cite[Thm~3.8]{green_generic_initial_ideals}
      provides $\dim A_{i} = i+1$ for $i\ge d$.
   Thus the Hilbert polynomial of $A$ is $i+1$, that is $\ccI$ defines a subscheme of dimension $1$ 
      (the degree of the Hilbert polynomial) and degree $1$ (its leading coefficient).
   That is $R$ is a union of a line and a finite subscheme. 
   But since the Hilbert polynomial of the line is already $i+1$,
     the constant coefficient in the Hilbert polynomial of $A$ determines that the finite subscheme is redundant
     (contained in the line). Thus $R$ is a line.
\end{prf}

\subsection{Cubics in three variables are tame}\label{sect_cubics_in_3_variables}

Let $f \in S^3V$,  where $\dim V=3$. 
It is a classical statement known for at least 100 years, that $\sigma_4(v_3(\PP^2)) = \PP^8$.
This can be calculated explicitly, and 
  it also follows from the Alexander-Hirschowitz theorem,
  see e.g.~\cite[Thms~1.1 \& 1.2]{brambilla_ottaviani_on_AH_theorem}.
Thus $\borderrk(f) \leq 4 =3+1$, and thus $sr(f) = \borderrk(f)$ by Principle~\ref{prin_tameness}.
Since all schemes in $\PP^2$ are smoothable, also $cr(f) = sr(f)$.

\subsection{Cubics in four variables are tame}
Let $f \in S^3V$,  where $\dim V=4$, and suppose $f$ is $4$-concise.
By \cite[Cor.~2.6]{casnati_notari_irreducibility_Gorenstein_degree_9}
   all finite Gorenstein schemes in $\PP^3$ are smoothable, hence  $cr(f)=sr(f)$ by \cite[Lem.~2.3]{nisiabu_jabu_cactus}.
In this subsection we prove that also in this case $\borderrk(f) = sr(f)$.  

We have $\sigma_5(v_3(\PP^3)) = \PP(S^3 V)$ 
  (again, either calculate explicitly or just crack the nut with the sledgehammer of Alexander-Hirschowitz theorem, 
   see, e.g., \cite[Thms~1.1 \& 1.2]{brambilla_ottaviani_on_AH_theorem}).
By Principle~\ref{prin_tameness}, if $\borderrk(f) \le 4$, then $sr(f) = \borderrk(f)$, so we assume $\borderrk(f)=5$.
Thus $f$ is a limit of $f_t$, with $f_t \in \langle v_3(R_t)\rangle$, and for $t\ne 0$ the scheme $R_t \subset \PP V$
  is a reduced union of $5$ distinct points.
We let $R_0$ be the scheme, which is the flat limit of $R_t$ (the limit in the Hilbert scheme).
If $\dim \langle v_3 (R_0) \rangle  = 4$ (the expected dimension), then 
  $\lim_{t \to 0} \langle v_3(R_t)\rangle = \langle v_3(R_0)\rangle$, and 
  $f \in \langle v_3(R_0)\rangle$, that is $sr(f) \le 5$, so $\borderrk(f) = sr(f)$. 
 
Thus assume  $\dim \langle v_3 (R_0) \rangle  \le 3$, which is only possible, 
  if $R_0$ is contained in a line $\PP^1 \subset \PP V$ by Lemma~\ref{lem_contained_in_a_line}.
Thus the saturated ideal $\ccI^{sat}$ of $R_0$ is generated by $2$ linear equations and $1$ quintic equation. 
Let $\ccI_t \subset \Sym V^*$ for general $t\ne 0$  be the saturated ideal of $R_t$,
  and let $\ccI$ be the flat limit of ideals $\ccI_t \to \ccI$, so that $\ccI$ defines $R_0$,
  but is not necessarily saturated (with our assumptions we can even observe $\ccI$ is never saturated).
Since $(\ccI_t)_3 \subset \Ann(f_t)_3$ for $t \ne 0$, 
  we must have the limiting statement $\ccI_3 \subset \Ann(f)_3$.
So by \cite[Prop.~3.4(iii)]{nisiabu_jabu_cactus} we have $\ccI \subset \Ann(f)$. Furthermore $\ccI \subset \ccI^{sat}$.

Since $f$ is concise, the Hilbert function
\begin{align*}
\text{of } A_f = \Sym(V^*)/\Ann(f)& \text{ is } (1,4,4,1,0,0\dotsc),\\
\text{of } \Sym(V^*)/ \ccI^{sat}  & \text{ is } (1, 2, 3, 4, 5, 5, 5, \dotsc),\\
\text{of } \Sym(V^*)/ \ccI        & \text{ is either} (1,4,5,5,\dotsc) \text{ or } (1,4,4,5,\dotsc).
\end{align*}
The last Hilbert function is because the ideal $\ccI$ arises as a flat limit of $\ccI_t$, 
  which are saturated ideals of $5$ distinct points,
  and the Hilbert function of $\Sym(V^*)/\ccI$ must be the same as the Hilbert function of $\Sym(V^*)/\ccI_t$,
  and is bounded from above by $5$.

We now look at the ideal $\ccI_{\le 3}$, generated by the second and third degrees of $\ccI$ ($\ccI$ has no linear generator).
Note that this ideal defines a subscheme of $\PP^3$ containing a projective line $\PP^1$, the same as $\langle R_0 \rangle$.
This is because you need at least a quintic to define $R_0$ inside the $\PP^1$.
We have $\ccI_{\le3} \subset \Ann(f)$ and $\Ann(f)$ needs at least two more minimal generators than those present in $\ccI_{\le 3}$.
Otherwise, $\Ann(f)$ would define a non-empty scheme, a contradiction.
Comparing the Hilbert functions, we see that at most one of these generators is a quadric.
Thus there is at least one minimal generator of $\Ann(f)$ of degree $\ge 3$.

For cubics, $\Ann(f)$ is generated in degrees at most $4$, and there is a generator of degree $4$ 
  if and only if $f$ is of rank $1$, \cite[Remark 4.3]{casnati_notari_irreducibility_Gorenstein_degree_10} or \cite[Proposition~6]{nisiabu_jabu_kleppe_teitler_direct_sums}.
So in our case, $\Ann(f)$ is generated in degrees at most $3$, 
  and thus $\Ann(f)$ has at least one minimal generator in degree $3$.
Polynomials for which the annihilator has a minimal generator in degree equal to the degree of the polynomial are studied 
  in details in \cite{nisiabu_jabu_kleppe_teitler_direct_sums}.
This special case has been studied earlier by Casnati and Notari, 
   \cite[Lem.~4.5]{casnati_notari_irreducibility_Gorenstein_degree_10} .
In particular, up to a linear choice of coordinates $(x,y_1, y_2, y_3)$, 
  $f$ is either $x^3 + g(y_1, y_2, y_3)$ or $xy_1^2 + g(y_1, y_2, y_3)$ 
  for some cubic $g \in S^3 \langle y_1, y_2, y_3\rangle$.
  In either case, $sr(g) \le 4$ by Section~\ref{sect_cubics_in_3_variables}.

If $f = x^3 + g(y_1, y_2, y_3)$, then $sr(f) \le sr(g) +1 \le 5$.
Since $\borderrk(f)=5$ and $\borderrk(f) \le sr(f)$, we have $sr(f) = 5$ and $\borderrk(f) = sr(f)$.

If  $f = xy_1^2 + g(y_1, y_2, y_3)$, then we can further change variables and replace $x$ with $\tilde{x}$ 
  and suppose $\tilde{x} = x + a_1 y_1 + a_2 y_2 + a_3 y_3$,
  so that $f = \tilde{x} y_1^2 + \tilde{g}(y_1, y_2, y_3)$, where $\tilde{g}$ has no terms $y_1^3$, $y_1^2y_2$, $y_1^2y_3$.
That is $\tilde{g}$ is singular at $(1:0:0)$.
Singular cubics in $3$ variables have border rank (and so also smoothable rank) at most $3$,
  see \cite[Table~1, p.353]{landsberg_teitler_ranks_and_border_ranks_of_symm_tensors}.
Since $sr(\tilde{x} y_1^2) =2$, we have $sr(f) \le 5$. Thus $sr(f) = 5$ and $br(f) = sr(f)$.

\section{Wild cases} 

Let $V:=k^5$ with basis $\set{x_0,x_1,y_0,y_1,y_2}$
  and let 
  \[
    \Sym V^* =k[\alpha_0,\alpha_1,\beta_0,\beta_1,\beta_2]
  \]
be the dual ring.
We will consider a specific $f \in S^3 V$, whose border rank is $5$,
cactus rank and smoothable rank are $6$, and rank is $9$.

Take
\begin{align*}
f &=  x_0^2\cdot y_0 - (x_0 + x_1)^2 \cdot y_1 +  x_1^2\cdot y_2\\
  &= x_0^2\cdot y_0 - x_0^2\cdot y_1 - 2\cdot x_0\cdot x_1\cdot y_1 - x_1^2\cdot y_1 + x_1^2\cdot y_2.
\end{align*}

We will prove the claims about each rank separately in the following subsections.

\subsection{Border rank}
Let
\begin{align*}
  p_1 & := x_0^3,         &  v_1 & :=    x_0^2\cdot y_0, \\ 
  p_2 & := (x_0+x_1)^3,   &  v_2 & := - (x_0 + x_1)^2 \cdot y_1,\\
  p_3 & := x_1^3,         &  v_3 & :=           x_1^2\cdot y_2,\\
  p_4 &:=(x_0 - x_1)^3,   &  v_4 & := 0,\\
  p_5 &:=(x_0+ 2 x_1)^3,  &  v_5 & := 0.\\
\end{align*}
Then $\fromto{[p_1]}{[p_5]} \in X$ and these points are linearly dependent.
Also $v_i \in \hat T_{p_i} X$, so by Proposition~\ref{prop_r_linearly_dependent points} 
the polynomial $f= v_1 +\dotsb + v_5$ has border rank at most $5$. 
Explicitly, if 
\begin{multline*}
f_t:= \tfrac 1 3 \cdot(x_0+t\cdot y_0)^3 - \tfrac 1 3 \cdot((x_0+x_1)+t\cdot y_1)^3 + \\
- \tfrac{1}{12} \cdot (2x_1-t\cdot y_2)^3 - \tfrac 1 9 \cdot (x_0-  x_1)^3+\tfrac 1 9 \cdot (x_0+ 2 x_1)^3
\end{multline*}
then $f= \lim_{t \to 0} \frac{1}{t} f_t$.

\begin{rem}
   The polynomial $f$ seems to be very special among those that appear in Proposition~\ref{prop_r_linearly_dependent points}
      with $r=5$. But in fact (up to a choice of coordinates) it is a general polynomial of this type.
\end{rem}

\subsection{Cactus rank}

To show that the cactus rank of $f$ is at least $6$ we will prove that there does not exist a scheme $R \subset \PP^4$ of length at most $5$ 
   such that $f \in \langle v_3(R) \rangle$.
Suppose on the contrary, there is such a scheme $R$.
We have
\[
\Ann(f) \supset \ccI(R)
\]
The Hilbert function of the algebra $\Sym V^* / \Ann(f)$ is $(1,5,5,1,0,0,\dots)$ by a direct calculation:
  this algebra has a symmetric Hilbert function (see \cite[Prop.~3.4(v)]{nisiabu_jabu_cactus}),
  and we compute $\ker(S^1 V^* \rightarrow S^2 V)$ to get
\begin{align*}
\Ann(f)_2&=\langle (\beta_0,\beta_1,\beta_2)^2, \phi_1,  \phi_2,  \phi_3,  \phi_4 \rangle, \text{ where}\\
 \phi_1&= \alpha_1\cdot \beta_0, \\
 \phi_2&= \alpha_0\cdot \beta_2, \\
 \phi_3&= -\alpha_0\cdot \beta_1 + \alpha_1\cdot \beta_1, \\
 \phi_4&= \alpha_0\cdot \beta_0 + \alpha_0\cdot \beta_1 + \alpha_1\cdot \beta_2.
\end{align*}
We must have $\ccI(R)_2 = \Ann(f)_2$, 
  because on the one hand $\ccI(R) \subset \Ann(f)$, 
  and on the other hand the length of $R$ is at most $5$,
  so that 
\[
     \dim (\Sym V^*/\ccI(R))_2 \le 5 = \dim (\Sym V^* / \Ann(f))_2.
\]
This implies $R \subset Z(\Ann(f)_2)$ and so $\ccJ \subset \ccI(R)$,
   where $\ccJ$ is the saturation of $(\Ann(f)_2)$.
A direct computation shows that   $(\Ann(f)_2)$ is not saturated. 
Namely, all the $\beta_i$'s are in its saturation. 
As an example, let us see that $\beta_0 \in \ccJ$.
First $\langle \alpha_1, \beta_0, \beta_1, \beta_2\rangle \cdot \beta_0 \subset \Ann(f)_2$.
Also, we can express
$\alpha_0^3\cdot \beta_0$ as
\[
(\alpha_0^2- \alpha_0\cdot\alpha_1)\cdot\phi_4 - (\alpha_0\cdot\alpha_1-\alpha_1^2)\cdot\phi_2+\alpha_0^2\cdot\phi_3+\alpha_0^2\cdot\phi_1
\]
where the $\phi_i$'s are  the last four generators of $\Ann(f)_2$, as defined above.

Since $\ccJ \subset \ccI(R) \subset \Ann(f)$, we have an inequality of the Hilbert functions
\[
  H(\Sym V^* \slash \ccJ) \ge H(\Sym V^* \slash  \ccI(R)) \geq H ( \Sym V^* \slash \Ann(f)).
\]
But we know that $H ( \Sym V^* \slash \Ann(f)) (1) = 5$, whereas $H ( \Sym V^* \slash \ccJ) (1) = 2$, a contradiction.

As a conclusion $cr(f) = 6$,
  because $f$ is in the span of three disjoint double points, which is a scheme of length $6$.

\subsection{Rank}  

The following lemma explains that whenever we have an infinite family of decomposable polynomials in a radical ideal, 
  then we can find  more decomposable polynomials in this ideal.
The geometric meaning of the lemma is the following. 
Suppose $R \subset V$ is a reduced subscheme,
  such that
\[
   R \subset H_{a} \cup  H'_{a}
\]
for  hypersurfaces $H_{a}$, $H'_{a}$ depending algebraically on a parameter $a \in C$ (with $C$ of positive dimension).
Then 
\[
   R \subset \left(\bigcap_{a\in C} H_{a}\right) \cup  \left(\bigcap_{a'\in C} H_{a'}\right)
\]

\begin{lemma}\label{lem_ideal_that_has_many_products}
  Fix an integer $i$.
  Suppose $I \subset \Sym V^*$ is a radical ideal
    and suppose $C$ is a positive dimensional (irreducible) algebraic variety over $k$.
  Suppose we have two rational maps $\beta, \gamma\colon C \dashrightarrow S^{\le i} V^*$,
    $a \mapsto \beta_a , \gamma_a$, 
    such that $\beta_a \gamma_a \in I$.
  Then for all $a, a'\in C$, we have $\beta_a \gamma_{a'} \in I$, 
    whenever $\beta_a$ and $\gamma_{a'}$ are defined.
\end{lemma}

\begin{proof}
  Let $R \subset \PP(V)$ be the algebraic set defined by $I$, so that $I=I(R)$.
  Fix any point $r \in R$.
  Let $U_{\beta} \subset C$ be the open subset of those points $a\in C$, where $\beta$ is defined and $\beta_a(r) \ne 0$.
  Define $U_{\gamma}$ analogously.
  
  Suppose on the contrary to the claim of the lemma, 
    that there exists $a, a'$, such that
    $\beta_a \gamma_{a'}(r) \ne 0$ for some $r \in R$.
  That is $\beta_a(r) \ne 0$ and $\gamma_{a'}(r) \ne 0$,
    and thus $U_{\beta}, U_{\gamma} \ne \emptyset$.
  Since both these sets are open and $C$ is irreducible, it follows that
    $U_{\beta} \cap U_{\gamma} \ne \emptyset$.
  Thus there exists $a \in C$, such that $\beta_a \gamma_a (r) \ne 0$,
    a contradiction with our assumption concerning~$I$.
\end{proof}

The polynomial $f$ is a sum of three cubic polynomials, each of a form $z^2w$, so each has border rank $2$ and rank $3$, i.e.,
  $r(f)\le 3 \cdot 3=9$. 
We claim $r(f) =9$.
Suppose by contradiction $R \subset \PP V$
  is a reduced subscheme of length at most $8$ such that $[f] \in \langle v_3(R) \rangle$.
Then, the defining ideal of $R$ satisfies $I(R) \subset \Ann(f)$,
  so, in particular, there are no linear forms in $I(R)$, and there are no squares among the quadratic forms in $I(R)$.
Moreover, there are at least $7$ quadrics in the ideal, because $7=15-8$, and there are 15 quadrics altogether, and the length of $R$ is at most $8$.
Since $\Ann(f)_2$ is $10$-dimensional
it follows that $\codim_k\bigl(I(R)_2 \subset \Ann(f)_2\bigr) \leq 3$.

By a direct computation we can check that 
\[
  (a_0\beta_0+a_1\beta_1+a_2 \beta_2)(c_0\alpha_0 + c_1\alpha_1) \in \Ann(f)
\]
has a solution for all $[a_0,a_1,a_2]=a \in C \subset \PP^2$ in a plane conic $C$.
Denote by $\beta_{a}:=a_0\beta_0+a_1\beta_1+a_2 \beta_2$ for any $a \in C$.
We already saw that $\langle (\beta_0,\beta_1,\beta_2)^2 \rangle  \subset \Ann(f)$.
Thus for each $a \in C$,  
  there is a four dimensional family of type $\beta_{a} \cdot \gamma \subset \Ann(f)$,
  where $\gamma$ can be a linear form (i.e., a $k$-linear combination of $\alpha_0, \alpha_1, \beta_0,\beta_1,\beta_2$). 
As $I(R)_2$ has codimension at most three in the annihilator, this family intersects non-trivially $I(R)_2$.
That is, for each $a \in C$, there exists a non-zero $\gamma_{a} \in V^*$ such that:
\[
  \beta_{a}\cdot \gamma_{a} \in I(R).
\]
Thus by Lemma~\ref{lem_ideal_that_has_many_products}:
\[
  \beta_{a}\cdot \gamma_{a'} \in I(R)
\]
for any $a, a' \in C$.
In particular, fixing $a'$, and considering linear combinations
  we  obtain $\langle\beta_0, \beta_1, \beta_2 \rangle \gamma_{a'} \subset I(R)$.
Since $I(R) \subset \Ann(f)$,
  this is only possible for $\gamma_{a'} \in \langle\beta_0, \beta_1, \beta_2 \rangle$.
But then $\gamma_{a'}^2 \in I(R)$, which contradicts the radicality of $I(R)$.  

This completes the proof of Theorem~\ref{thm_wild_and_tame_cubics} in the case $\dim V=5$.

\subsection{Higher dimensions}\label{sect_higher_dimension_examples}

A naive way to see that Theorem~\ref{thm_wild_and_tame_cubics} 
    also holds in the case $\dim V \ge 6$, is to consider the same $5$-concise polynomial 
    and apply \eqref{equ_calculate_ranks_for_concise}.
Alternatively, let $V = \langle x_0, x_1, y_0, y_1, y_2 \rangle \oplus W$ and pick any $\dim W$-concise $g \in S^3 W$,
    such that the border rank of $g$ is $\dim W$. We claim that $f + g$ is ``wild'',
    i.e. $\borderrk(f+g) = \dim V = 5 + \dim W$ and $sr(f+g) \ge cr(f+g) \ge \dim V +1 = 6 + \dim W$.
The argument is identical to the one above, using $\Ann(f+g)_2 = \Ann(f)_2 \cap \Ann(g)_2$.

\bibliography{rank_example}
\bibliographystyle{alpha_four}
\end{document}